\documentclass[12pt,a4paper,oneside,reqno]{article}

\usepackage[utf8]{inputenc}
\usepackage[english]{babel}
\usepackage[left=1in,right=1in,top=1in,bottom=1in]{geometry}
\usepackage{amsfonts,amssymb,amsmath,amsthm}
\usepackage[symbol]{footmisc}
\usepackage{longtable,array}
\newcolumntype{L}[1]{>{\raggedright\arraybackslash$}p{#1}<{$}}
\usepackage{enumerate}
\usepackage{cite}
\usepackage{xcolor}

\usepackage{hyperref}
\hypersetup{hidelinks}
\usepackage{fancyhdr}
\theoremstyle{plain}
\newtheorem{theorem}{Theorem}

\theoremstyle{definition}

\makeatletter
\@removefromreset{theorem}{section}\@removefromreset{equation}{section}

\makeatother
\usepackage{adjustbox}
\usepackage{tikz}\usetikzlibrary{arrows}

\IfFileExists{fouriernc.sty}{\usepackage{fouriernc}}{}

\newcommand{\doi}[1]{\href{https://doi.org/#1}{doi:\detokenize{#1}}}
\newcommand{\arxiv}[1]{\href{https://arxiv.org/abs/#1}{arXiv:#1}}

\theoremstyle{plain}
\newtheorem{Th}[theorem]{Theorem}
\newtheorem{Lem}[theorem]{Lemma}
\newtheorem{pro}[theorem]{Proposition}
\newtheorem{cor}[theorem]{Corollary}
\theoremstyle{definition}

\theoremstyle{remark}
\newtheorem*{Rem}{Remark}
\newcommand{\0}{_{\bar 0}}
\newcommand{\1}{_{\bar 1}}
\def\N{\mathbb N}

\begin{document}

\noindent\fbox{\parbox{\dimexpr\textwidth-2\fboxsep-2\fboxrule\relax}{\small
\textbf{Corrected version (verified by Claude).} This is a corrected version of the paper I.~Kaygorodov, Yu.~Popov, and A.~Pozhidaev, \emph{The universal conservative superalgebra}, Communications in Algebra, 47 (2019), no.~10, 4066--4076, \doi{10.1080/00927872.2019.1576189}.}}

\bigskip\bigskip

\begin{center}
{\Large The universal conservative superalgebra}

\bigskip
\bigskip
{\bf Ivan Kaygorodov, Yury Popov, and Alexandr Pozhidaev}
\end{center}

\bigskip\bigskip
\noindent{\small\textbf{Abstract.} We introduce the class of conservative superalgebras, in particular, the superalgebra $\mathcal{U}(V)$ of bilinear operations on a superspace $V$. Moreover, we show that each conservative superalgebra modulo its maximal Jacobi ideal is embedded into $\mathcal{U}(V)$ for a certain superspace $V$.}



\bigskip
\bigskip
\tableofcontents

\bigskip

\newpage 
\section*{Introduction}\addcontentsline{toc}{section}{Introduction}

I.\,Kantor introduced the class of conservative algebras in \cite{Kantor72}.
This class includes some well-known classes of algebras,
such as associative, Jordan, Lie, Leibniz, and Zinbiel algebras \cite{KLP}.

To define conservative algebras  we firstly introduce some notations. Let $V$ be a vector space over a field $\mathbb{F},$ let $\varphi$ be a linear map on $V$, and let $B$ be a bilinear map on $V$ (i.~e., an algebra). Then we can consider a product of $\varphi$ and $B$, which is a bilinear map $[\varphi,B]$ on $V$ given by
\begin{equation}
\label{lin_bilin_act}
[\varphi,B](x,y)= \varphi(B(x,y))-  B(\varphi(x),y)-B(x,\varphi(y)).
\end{equation}
Note that this product measures how far $\varphi$ is from being a derivation of the algebra $B$.
The relation (\ref{lin_bilin_act}) may also be considered as a transformation of a bilinear operator $B$ under the action of an infinitesimal transformation $x\mapsto x+t\varphi(x)$. Indeed, the right-hand side of (\ref{lin_bilin_act}) is the coefficient at the first degree of $t$ in the series $e^{\varphi t}(B(e^{-\varphi t}(x),e^{-\varphi t}(y)))$.
Thus, $\{[L_a,B]:a\in V\} $ is the set of all algebras which arise from the initial algebra $B$ by the action of left shifts $L_a,\ a\in V$. Thus, the definition of a conservative algebra given below says that this set (for every conservative algebra $B$) is transformed into itself under other actions of the left shifts $L_a,\,a\in V$.\smallskip

An algebra $A$ with a multiplication $\cdot$ is called a ({\it left}) {\it conservative algebra}  if there exists an  algebra structure $* $ (called an {\it associated algebra}) on the underlying space of $A$ such that
\[[L_b,[L_a, \cdot]]=- [L_{a*b},\cdot];\]
here, as usual,  $L_a$ stands for the operator of left multiplication by $a$: $L_a(x):=a\cdot x:=ax$ for all $x\in  A$ (in what follows, the symbol $:=$ denotes an equality by definition; and $\langle
\Upsilon \rangle$ stands for the linear span of a set $\Upsilon$ over the ground field). Replacing the left multiplications with the right multiplications and modifying correspondingly the above relation, we can define right conservative algebras and obtain a similar theory. The associative and Lie algebras give obvious examples of conservative algebras (see below for the supercase).

In the theory of conservative algebras, the algebra $\mathcal{U}(n)$, which was introduced in \cite{Kantor90}, is of great importance. Let $V_n$ be an $n$-dimensional vector space over $\mathbb{F}.$ The space of the algebra $\mathcal{U}(n)$ is the space of all bilinear operations on $V_n$ (further we identify a bilinear operation $A \in \mathcal{U}(n)$ with the algebra structure that it defines on $V_n$, and do the same in the supercase). To define a product $\bigtriangleup$ on $\mathcal{U}(n)$ we fix a nonzero vector $u \in V_n.$ Then for $A, B \in \mathcal{U}(n)$ we put
\begin{eqnarray}\label{Kanpro}
(A \bigtriangleup_u B)(x,y) = [L_u^A,B](x,y) =
A(u,B(x,y))-B(A(u,x),y)-B(x,A(u,y)),
\end{eqnarray}
where $L_u^A: x \mapsto A(u,x)$ is the left multiplication with respect to $A.$ This product is called the \textit{Kantor product of $A$ and $B$.}
One can easily check that the algebras obtained by different choices of nonzero $u \in V_n$ are isomorphic (see the proof for the superalgebra case below). Note also that  (\ref{lin_bilin_act}) gives a Lie action of $\mathfrak{gl}_n$ on $\mathcal{U}(n),$ since it coincides with the natural action of $\mathfrak{gl}_n$ on $\mathcal{U}(n) = V_n^*\otimes V_n^* \otimes V_n.$

One can verify that the algebra $\mathcal{U}(n)$ is conservative with the associated multiplication $\bigtriangledown$ given, for example, by $A\bigtriangledown_u B(x,y) = -B(u,A(x,y))$ (there are other associated multiplications as well). In the theory of conservative algebras, the algebra $\mathcal{U}(n)$ plays a role analogous to the role of $\mathfrak{gl}_n$ in the theory of Lie algebras, that is, every finite-dimensional conservative algebra (modulo its maximal Jacobian ideal) may be embedded into $\mathcal{U}(n)$ for some $n.$ Some properties of the algebra $\mathcal{U}(2)$ were studied in \cite{KLP,KV15}.

The Kantor product of a multiplication $\cdot$ by itself is
called its {\it Kantor square.} It gives a map $K$ from any variety $\mathcal{V}$ of algebras to some class of algebras $K(\mathcal{V}).$
 The Kantor squares of multiplications satisfying certain
 conditions (such as the associativity, the Jacobi identity and others)
 were studied in \cite{Kg17}.

The main aim of this paper is to introduce the conservative
superalgebra $\mathcal{U}(n,m),$ which is the super-counterpart of the algebra
$\mathcal{U}(n)$, and prove that every finite-dimensional conservative
superalgebra is embedded  (modulo its maximal Jacobian ideal) into
$\mathcal{U}(n,m)$ for some non-negative integers $n,m.$

\section{The Conservative Superalgebras}
 In this section we introduce the class of conservative superalgebras, whose definition is a complete analogue of the notion of a conservative algebra given in \cite{Kantor72, Kantor90}, and we also prove some of their
 elementary properties, which are used further.

\subsection{Notation and Main Definitions} As usual, all (super)spaces and (super)algebras are considered over
a field $\mathbb{F}.$ Algebras and superalgebras are in general
assumed to be nonassociative, noncommutative, and without unity. Let
$A = A_{\bar{0}} \oplus A_{\bar{1}}$ be a superalgebra
 (i.~e., it is a ${\mathbb Z}_2$-graded algebra:
 $A_{\bar{i}} A_{\bar{j}}\subseteq A_{\overline{i+j}}$, the elements in
 $A_{\bar{0}}$ are {\it even}, and the elements in $A_{\bar{1}}$ are {\it odd}),
$(-1)^{xy}:=(-1)^{p(x)p(y)}$, where $p(x)$ is the parity of $x$,
that is,   $p(x)=i$ if $x\in A_{\bar{i}}$. The elements in
 $A_{\bar{i}}$
 are {\it homogeneous}. We use the same notation
 for a homogeneous operator $\varphi$:
$(-1)^{x\varphi}:=(-1)^{p(x)p(\varphi)}$ (here $p(\varphi)$ is the
parity of $\varphi$:
 $\varphi(A_{\bar{i}})\subseteq A_{\overline{i+p(\varphi)}}$ ),
 and so on.  In what follows, if the parity of
 an element (operator) appears in a formula, then this element
 (operator) is assumed to be homogeneous. Sometimes we simply write
 ``subspace'' instead of ``subsuperspace'', ``subalgebra''
 instead of ``subsuperalgebra'', and so on.
  The ideals are assumed to be homogeneous, i.~e., an ideal $I$
  contains with
  every element $x=x_{\bar{0}}+x_{\bar{1}}\in I$
  its even and odd components
  $x_{\bar{i}}\in I\cap A_{\bar{i}},\ i=0,1$.

Following \cite{Kantor72, Kantor90},
 we define some supercommutators. Let
 $V$ be a vector superspace, let $A$ be a linear operator on $V,$ and let $B$ and $C$ be bilinear operators on $V.$ For all $x,y,z \in V,$ put
\begin{gather}
\label{oper_comm1} [A,x]=A(x), \ [B,x](y)=B(x,y),\\
\label{oper_comm2} [A,B](x,y)=
A(B(x,y))- (-1)^{BA} B(A(x),y)-(-1)^{A(B+x)}B(x,A(y)),\\
 [B,C](x,y,z)=B(C(x,y),z)+ (-1)^{xC}B(x,C(y,z))+ (-1)^{y(C+x)}B(y,C(x,z)) \nonumber\\
-(-1)^{BC}C(B(x,y),z)- (-1)^{B(C+x)}C(x,B(y,z))-(-1)^{CB+xy+By}C(y,B(x,z)). \nonumber
\end{gather}
By definition we put $[X,Y]=-(-1)^{XY}[Y,X]$ for all $X,Y\in \{a,A,B\}$.

Denote by $\mathcal{U}(V)$ the space of all bilinear operations on
$V.$ Then one can verify that (\ref{oper_comm2}) defines an action
of $\mathfrak{gl}(V)$ on $\mathcal{U}(V).$

Let $M(x,y)=xy$ be an even bilinear operation that defines a
superalgebra structure on $V.$ Let $L_a$ be as above. We say that
$M$ is a {\it conservative superalgebra} if there exists an even
superalgebra $M^*(x,y) = x*y$ (called the {\it associated
superalgebra}) on the same space such that
\begin{equation}
\label{cons_oper}
[L_b, [L_a,M]]=- (-1)^{ab}[L_{a*b},M]
\end{equation}
for all $a,b\in V$.

The above relation can be written explicitly as an identity of
degree 4 with respect to the multiplications $M$ and $M^*:$
\begin{multline*}
b(a(xy))- b((ax)y)- (-1)^{ax} b(x(ay))-(-1)^{ab}a((bx)y) +(-1)^{ab}(a(bx))y+\\
(-1)^{ax}(bx)(ay)-(-1)^{b(a+x)}a(x(by))+(-1)^{b(a+x)}(ax)(by)+(-1)^{x(a+b)+ab} x(a(by))=\\
-(-1)^{ab} (a*b)(xy)+(-1)^{ab}((a*b)x)y +(-1)^{x(a+b)+ab}x((a*b)y).
\end{multline*}

One may also use the general approach to define the conservative
superalgebras. Namely, let $\Gamma:=\Gamma\0\oplus\Gamma\1$
  be the Grassmann superalgebra in generators
  $1,\ \xi_i,\ i\in \N$, $\Gamma\0=\left<1,\xi_{i_1}\ldots \xi_{i_{2k}}:k\in
  \N\right>$, $\Gamma\1=\left<\xi_{i_1}\ldots \xi_{i_{2k-1}}:k\in
  \N\right>$. Let ${\cal A}:={\cal A}\0\oplus{\cal A}\1$ be a superalgebra and $\cdot$ and $*$ be two products on $\mathcal{A}.$ Consider its {\it Grassmann envelope}
 $\Gamma({\cal A}):=
 ({\cal A}\0\otimes\Gamma\0)\oplus({\cal A}\1\otimes\Gamma\1)$ and extend the products $\cdot$ and $*$ to $\Gamma({\cal A})$ as follows:
 \[(a\otimes f)\cdot(b\otimes g)=(-1)^{ab}ab\otimes fg,\]
 \[(a\otimes f)*(b\otimes g)=(-1)^{ab}a*b\otimes fg\]
 for all homogeneous $a,b\in {\cal A},f,g\in \Gamma\ (p(a)=p(f),\,
 p(b)=p(g)).$ Then $(\mathcal{A},\cdot)$ is conservative with an associated multiplication $*$ if and only if $(\Gamma(\mathcal{A}),\cdot)$ is a conservative algebra with an associated multiplication $*.$

\subsection{Examples}\label{subsec_examples} The Lie superalgebras give obvious examples of conservative
 superalgebras.
 Indeed, let $L$ be a Lie superalgebra with a product $M$. Then
   the Jacobi identity and the anticommutativity imply that
   $[L_a,M]=0$ for all $a\in L$. Thus, the left and right-hand sides
   of (\ref{cons_oper}) are zero for arbitrary product $M^*$ on $L$.
   As another example we have associative superalgebras.
   In this case
   \[[L_a,M](x,y)=-(-1)^{ax}xay;\quad [L_b,[L_a,M]](x,y)=(-1)^{b(a+x)+ax}xaby,\]
    and (\ref{cons_oper}) holds with $x*y:=xy$.\smallskip

 A linear space $U$ with a bilinear operation
  $M:U\times U\mapsto U$ is called a {\it terminal algebra}
  provided that
  \begin{eqnarray}\label{termal}
[[[M,a],M],M] &=& 0
\end{eqnarray}
for every $a\in U$. Note that (\ref{termal}) is an identity of degree 4, therefore, the class of terminal algebras is a variety. The class of terminal algebras is vast: it includes, for example, Jordan algebras, Lie algebras and (left) Leibniz algebras. In \cite{Kantor66, Kantor72} it was shown that the commutative algebras satisfying
 (\ref{termal}) are Jordan algebras. Assuming that in (\ref{termal}) we
 have the supercommutators, we arrive at the definition of
 terminal superalgebra. Passing to the supercase we see that
 the commutative superalgebras satisfying
  (\ref{termal}) are Jordan superalgebras.

The following result can be obtained by a direct computation:

\begin{pro} Let $\operatorname{char}\mathbb{F} \neq 3.$ An algebra $M$ is terminal if and only if it is conservative
 and the multiplication in the associated superalgebra $M^*$ can be defined by
\[M^*(x,y)=\frac{2}{3}xy+\frac{1}{3}yx.\]
\end{pro}

The following theorem provides us with different
 examples of conservative superalgebras.

\begin{Th}
\label{alphabeta}
Let ${\cal V}$ be a homogeneous variety of algebras. Assume that
there exist $\alpha, \beta \in \mathbb{F}$
 such that every ${\cal V}$-algebra is conservative with
  the associated multiplication given by the rule
 $a\,*\,b~=~\alpha ab~+~\beta ba.$ Then every ${\cal V}$-superalgebra is
conservative with the associated multiplication  $a\,*\,b~=~\alpha
ab + (-1)^{ab}\beta ba.$
\end{Th}

\begin{proof} Let $M$ be a $\mathcal V$-superalgebra. Consider the Grassmann envelope $\Gamma(M)$, which is an algebra in $\mathcal V$. By our assumptions, $\Gamma(M)$ is a conservative algebra, and the multiplication in the associated algebra
 $\Gamma(M)^*$ can be defined
 by the formula
 $\Gamma(M)^*(x,y)=\alpha xy+\beta yx$
  for all $x,y\in \Gamma(M)$. This multiplication is obviously induced by a multiplication $M^*$ on the space of $M$ which is given by $M^*(x,y)=\alpha xy+(-1)^{xy}\beta yx.$
 Therefore, by our general approach above, $M$ is a conservative superalgebra with an associated multiplication $M^*.$ \end{proof}

It follows that associative, quasi-associative, Jordan, terminal, Lie, Leibniz, and Zinbiel superalgebras are conservative (see \cite{KLP}). In particular, a superalgebra $M$ is terminal if and only if it is conservative and the multiplication in the associated superalgebra $M^*$ can be given by
\begin{equation}
\label{terminal_assmult}
M^*(x,y)=\frac{2}{3}xy+(-1)^{xy}\frac{1}{3}yx.
\end{equation}

As we have seen, associative and Jordan superalgebras are conservative with the associated multiplication $M = M^*.$ It is natural to ask what is the subclass of conservative superalgebras with this additional restriction.

A superalgebra $U$ is called a
{\it noncommutative Jordan superalgebra}
if $U$ is flexible (that is, the operator identity $[R_x,L_y]=[L_x,R_y]$ holds in $U$) and its symmetrized superalgebra (the algebra on the space of $U$ with the multiplication $x \circ y = \frac{1}{2}(xy + (-1)^{xy}yx)$) is Jordan. For more information on noncommutative Jordan superalgebras see \cite{ps3, p1} and references therein.

\begin{pro} A flexible conservative superalgebra with
 the product $M$ whose associated superalgebra
 has the same product $M^*=M$ is a
 noncommutative Jordan superalgebra.
 \end{pro}

 \begin{proof} Let ${\cal A}$ be such flexible conservative superalgebra.
   Then $\Gamma({\cal A})$
 is a flexible conservative algebra with the extended products
 $M=M^*$. By \cite[Proposition 1]{Kantor72}, $\Gamma({\cal A})$
 is a  noncommutative Jordan algebra. Therefore, by the general
 definition of a superalgebra of a given variety ${\cal V}$, ${\cal
 A}$ is a  noncommutative Jordan superalgebra.  \end{proof}

\begin{pro}
A conservative superalgebra $M$ with a unity is a
noncommutative Jordan superalgebra. \smallskip\rm
\end{pro}

  \begin{proof} Let ${\cal A}$ be a conservative superalgebra with a unity.
   Then $\Gamma({\cal A})$
  is a conservative algebra with a unity.
  By \cite[Proposition 2]{Kantor72}, $\Gamma({\cal A})$
 is a  noncommutative Jordan algebra. Therefore, ${\cal
  A}$ is a  noncommutative Jordan superalgebra.  \end{proof}

\subsection{Operator Superalgebras}\label{subsec_operator}

Let $M$ be a conservative superalgebra on an underlying vector
space $V$. Considering both parts of (\ref{cons_oper}) as operators
acting on $y \in V,$ we obtain  the following operator relation:
\begingroup\small\[
[L_b,[L_a,L_x]] - [L_b,L_{ax}] - (-1)^{ab}[L_a,L_{bx}] +
(-1)^{ab}L_{a(bx)} +(-1)^{ab}[L_{a*b},L_x] - (-1)^{ab}L_{(a*b)x} = 0.
\]\endgroup{}
Therefore, $\mathcal{U}_0(V):= \langle L_a, [L_a,L_b] : a, b \in V
\rangle\leq \mathfrak{gl}(V)$. Moreover, since (\ref{oper_comm2})
gives an action of the Lie superalgebra $\mathfrak{gl}(V)$ on
$\mathcal{U}(V)$, we immediately get
\[[[L_b,L_a],M] = [L_{b*a - (-1)^{ab}a*b},M],\]
which implies that  $\mathcal{U}_1(V):= \langle M, [L_a,M]: a \in V
\rangle$  is a $\mathcal{U}_0(V)$-submodule of $\mathcal{U}(V).$ This also implies that the operators $[L_b,L_a]-L_{b*a - (-1)^{ab}a*b},\ a, b \in V,$ are superderivations of $M.$

\begin{cor} Let $M$ be a terminal superalgebra. The linear transformations
$[L_a,L_b]-\frac{1}{3}L_{[a,b]}$ are superderivations of $M$
for all $a,b\in M$.
\end{cor}

\begin{Rem}
Let $M$ be a (super)algebra such that $\mathcal{U}_1(M)$ is a
$\mathcal{U}_0(M)$-submodule of $\mathcal{U}(M)$. Then $M$ is
called \textit{rigid} or \textit{quasi-conservative}
\cite{Kantor89_L}. In \cite{kacan}, the simple linearly compact rigid
commutative and anticommutative superalgebras over an algebraically
closed field of characteristic 0 were classified, and in \cite{CKO}
the 2-dimensional rigid algebras were classified.
\end{Rem}

\subsection{Jacobi Elements and Quasiunities}

An element $a$ in a superalgebra $M$ is called a {\it Jacobi
element} provided that
\begin{equation}
\label{jacobi_1} a(xy)=(ax)y+(-1)^{ax}x(ay)
\end{equation}
holds for all $x,y\in M$.

In other words, $a$ is a Jacobi element if $L_a$ is a
superderivation of $M$. The relation (\ref{jacobi_1}) can be
rewritten in the following forms:
\begin{equation}
\label{jacobi_2} [L_a,L_x] = L_{ax} \text{ for every } x \in M,
\end{equation}
\begin{equation}
\label{jacobi_3}
[L_a,M] = 0.
\end{equation}

Denote by $J$ the space of all Jacobi elements of a superalgebra
$M$.  Let $N:=\{a \in M:L_a = 0\}$ be the {\it left annihilator} of
 $M$. Obviously, $N \subseteq J.$ An ideal
 $I$ of $M$ is called a {\it Jacobi ideal} provided that $I\subseteq
 J$.

The following statement is immediate from the definitions and (\ref{jacobi_2}):
\begin{Lem} Let $M$ be a superalgebra, and let $J$ and $N$ be as above.
 Then $J$ is a subsuperalgebra of $M;$
    $N$ is an ideal of $J$, and the quotient superalgebra $J/N$
    is isomorphic to a subsuperalgebra of the Lie superalgebra
    of derivations of $M.$
If $M$ possesses a unity then $J=0;$
  and if $M$ is a Lie superalgebra then $J=M$.
\end{Lem}

An even element $e \in M$ is said to be a {\it left quasiunity} if the equality
\[e(xy)=(ex)y+x(ey)-xy \]
holds for all $x, y \in M$. This condition is equivalent to the
relations
\begin{equation}
\label{quasiunit_oper} [L_e,L_x] = L_{ex-x} \textrm{ for every } x
\in M,
\end{equation}
\begin{equation}
\label{quasiunit_M}
[L_e,M] = -M.
\end{equation}
Obviously, a left unity is a left quasiunity. But, in general, the converse is not true (see examples in the next section).
The following theorem is proved similarly to
 \cite[Theorem 1]{Kantor90}, so we only give an outline of the proof.
\begin{Th}
Let $M$ be a conservative superalgebra. The associated superalgebra
$M^*$ is defined up to an arbitrary superalgebra with values in $J$.
Moreover, the following relations hold$:$
\begin{eqnarray}
M^*(a,b) \equiv 0 \ (\operatorname{mod} J),\ a \in J, \label{8} \\
M^*(a,b) \equiv -(-1)^{ab}ba  \ (\operatorname{mod} J),\ b \in J.
\label{9}
\end{eqnarray}
If $M$ has a left quasiunity $e,$ then
\begin{eqnarray}
M^*(e,a) \equiv a,\ M^*(a,e) \equiv 2a-ea \ (\operatorname{mod} J).
\label{10}
\end{eqnarray}
\end{Th}

\begin{proof}
The first statement and (\ref{8}) are immediate by
(\ref{cons_oper}) and (\ref{jacobi_3}). To prove (\ref{9}) it
suffices to show that
\begin{equation}
\label{jacobi_double_comm} [L_b,[L_a,M]] = [L_{ba},M]
\end{equation}
 for all $a \in M, b \in J$. It follows easily from (\ref{jacobi_2})
and (\ref{jacobi_3}) that
\[[L_b,[L_a,M]] = [[L_b,L_a],M] + (-1)^{ab}[L_a,[L_b,M]] = [L_{ba},M].\]
The first of the equations (\ref{10}) follows from (\ref{cons_oper})
and (\ref{quasiunit_M}). The second is proved by analogy with
(\ref{9}): using (\ref{quasiunit_oper}) and (\ref{quasiunit_M}) we
get
\begin{equation}
\label{quasiunit_double_comm}
[L_e,[L_a,M]] = [L_{ea-2a},M].
\end{equation}
\end{proof}

\section{The Universal Conservative Superalgebra}
In this section we define the superalgebra structure on the space
$\mathcal{U}(V)$ of all bilinear operations on a superspace $V$ and
prove that it is conservative. Moreover, we show that every
finite-dimensional conservative superalgebra is embedded (modulo its
maximal Jacobi ideal) in $\mathcal{U}(V)$ for a certain
finite-dimensional space $V.$

\subsection{The superalgebra $\mathcal{U}(V)$}
Let $V$ be a superspace. The space of the superalgebra
$\mathcal{U}(V)$ is the superspace of all bilinear operations
on $V.$
 Fix a nonzero homogeneous  $a\in V$. Define the multiplication
 $\bigtriangleup_a$ in $\mathcal{U}(V)$ by the rule
\begingroup\small\begin{eqnarray}
(A \bigtriangleup_a B)(x,y) = A(a,B(x,y))- (-1)^{\scriptscriptstyle
 B(A+a)}B(A(a,x),y)- (-1)^{\scriptscriptstyle (A+a)(B+x)}B(x,A(a,y)).
 \label{sukant}
\end{eqnarray}\endgroup{}

Consider the natural action of the group $\operatorname{gl}(V)$ of
even automorphisms of $V$ on $\mathcal{U}(V):$
\begin{equation}
\label{induced_iso} \varphi(A)(x,y) =
\varphi(A(\varphi^{-1}(x),\varphi^{-1}(y)))
\end{equation}
(we denote an automorphism and its action by the same symbol
$\varphi$). A direct computation shows that the mapping $A \mapsto
\varphi(A)$ is an isomorphism between
$(\mathcal{U}(V),\bigtriangleup_a)$ and
$(\mathcal{U}(V),\bigtriangleup_{\varphi(a)})$
Therefore, different nonzero
even (respectively, odd) vectors $a$ give rise to isomorphic {\it
even} (respectively, {\it odd}) superalgebras, which we denote
 $\mathcal{U}(V)^0$ and
 $\mathcal{U}(V)^1,$ respectively.

Moreover, consider the opposite superspace $V^{\Pi}$ given by
$V^{\Pi}_{\bar{0}} = V_{\bar{1}}, V^{\Pi}_{\bar{1}} = V_{\bar{0}}.$
Then the parity-reversing isomorphism $V \cong V^{\Pi}$ induces an
 isomorphism between $\mathcal{U}(V^{\Pi})^1$ and the odd
superalgebra obtained from $\mathcal{U}(V)^0$ by reversing the
parity. Therefore, it suffices to consider only the superalgebras
$\mathcal{U}(V)^0.$ For the sake of simplicity, we denote them
  by $\mathcal{U}(V).$

If $V = V_{n,m}$ is a finite-dimensional superspace with
 ${\rm dim} V_{\bar 0}=n$
and ${\rm dim} V_{\bar 1}=m$ (further in this case we say that $V$ is of
dimension $n+m,$) then we denote $\mathcal{U}(V)$ by
$\mathcal{U}(n,m).$

\begin{Th}
\label{U_cons} Let $V$ be a superspace, and let $a\in V_{\bar 0}$.
 The superalgebra $(\mathcal{U}(V),\bigtriangleup_a)$
is conservative, and the associated multiplication can be given by
\begin{equation}
\label{U_assmult_1}
A \bigtriangledown^1_aB(x,y)=-(-1)^{AB}B(a,A(x,y))
\end{equation}
or
\begin{equation}
\label{U_assmult_2}
A \bigtriangledown^2_aB(x,y)= \frac{1}{3}(A^*\bigtriangleup_aB + (-1)^{AB}\tilde{B}\bigtriangleup_aA),
\end{equation}
where $A^*(x,y)=A(x,y)+(-1)^{xy}A(y,x)$ and $\tilde{B}(x,y) = 2(-1)^{xy}B(y,x)-B(x,y).$
\end{Th}

\begin{proof}
Let $\bigtriangleup_a$ be the product on $\mathcal{U}(V)$
given by (\ref{sukant}). A straightforward
 computation shows that for $W, V \in
\mathcal{U}(V)$ and $x,y \in V$ we have

\begin{flushleft}$[L_A, \bigtriangleup_b](W,V)(x,y)=
(-1)^{(A+a)W} (W(A(a,b),V(x,y))-$\end{flushleft}
\begin{flushright} $- (-1)^{V(W+A+a+b)}V(W(A(a,b),x),y)-
(-1)^{(V+x)(W+A+a+b)}V(x,W(A(a,b),y)).$\end{flushright}

In other words,
\begin{eqnarray}[L_A, \bigtriangleup_b] (W,V) = (-1)^{AW} W
\bigtriangleup_{A(a,b)} V. \label{TAaa}
\end{eqnarray}

Now, for  $A \bigtriangledown^1_a  B(x,y)=-(-1)^{AB}B(a,A(x,y))$ we have
\[[L_B,[L_A, \bigtriangleup_a]](W,V)=(-1)^{(A+B)W} W
\bigtriangleup_{B(a,A(a,a))} V= -(-1)^{AB} [L_{A \bigtriangledown^1_a  B},
\bigtriangleup_a](W,V).\]

Analogously one can show that $(A\bigtriangledown^2_a B)(a,a) = -(-1)^{AB}B(a,A(a,a)),$ which proves that we can also take $\bigtriangledown^2_a$ as the multiplication in the associated superalgebra.
\end{proof}

By (\ref{jacobi_3}) and (\ref{TAaa}), the Jacobi subspace $J$ of
 $(\mathcal{U}(V),\bigtriangleup_a)$ consists precisely of those
$A(x,y) \in \mathcal{U}(V)$ for which $A(a,a)=0,$ so we may identify
the spaces $\mathcal{U}(V)/J$ and $V$ by the mapping $A \mapsto
A(a,a).$ In particular, for the algebra $\mathcal{U}(n,m)$ we have
$\operatorname{codim}(J) = n+m.$

If the mapping $\mathcal{U}(V)\to \mathfrak{gl}(V)$ given by $A \mapsto L_a^A$, is surjective
(which is always the case if $V$ is of countable dimension) then
every operator $A$ such that $L_a^A = -id$ is  a left unity of
$(\mathcal{U}(V),\bigtriangleup_a)$.

\begin{Lem}\label{lem_nojacobi}
The algebra $\mathcal{U}(V)$ has no nonzero Jacobi ideals.
\end{Lem}
\begin{proof}
By above, a bilinear operator $A$ lies in
the Jacobi subspace $J \subseteq  (\mathcal{U}(V),\bigtriangleup_a)$
if and only if $A(a,a) = 0.$ Therefore, for every $A$ in a Jacobi
ideal of $\mathcal{U}(V)$ and every $B \in \mathcal{U}(V)$ we have
\[0 = (A\bigtriangleup_a B)(a,a) = A(a,B(a,a)),\]
\begin{multline*}0 = (B\bigtriangleup_a A)(a,a) = -(-1)^{AB}A(B(a,a),a) -
(-1)^{AB}A(a,B(a,a))\\
= -(-1)^{AB}A(B(a,a),a),\end{multline*} whence $A(a,V) = A(V,a) = 0.$ Now, this
relation holds for $B\bigtriangleup_aA$ for every $B \in
\mathcal{U}(V):$
\[0 = (B\bigtriangleup_aA)(a,y) = -(-1)^{AB}A(B(a,a),y)\]
for all $y \in V$, and we get $A = 0.$
\end{proof}

\subsection{Terminal Subalgebras of $\mathcal{U}(n,m)$}

In this section we give some examples
 of terminal non-Jordan superalgebras, which are
 subsuperalgebras of $\mathcal{U}(n,m)$.
 These superalgebras are analogs of the simple terminal
 algebras  $W_n,\,S_n,\,H_n$ introduced in \cite{Kantor89}.\smallskip

{\bf\underline{The algebra $W_{n,\,m}.$}} The space of $W_{n,\,m}$
consists of
 all supersymmetric bilinear operations on a vector superspace $V$
 of dimension $n+m$.

It is easy to see that the algebra $M:=W_{n,\,m}$ is terminal: indeed, for supersymmetric operations $A, B$ the multiplication (\ref{U_assmult_2}) specializes exactly to (\ref{terminal_assmult}).

It follows from (\ref{jacobi_3}), (\ref{quasiunit_M}) and (\ref{TAaa}) that the Jacobi subspace
consists of the elements $A\in W_{n,\,m}$ such that $A(a,a)=0$, and
the left quasiunits satisfy the condition $A(a,a)=-a$.

Note also that
the algebra $W_{n,\,m}$ has left units; these are the elements $A\in
W_{n,\,m}$ for which $A(a,x)=-x$ for all $x$.\smallskip

{\bf\underline{The superalgebra $S_{n,\,m}$}} is the subsuperalgebra of $W_{n,\,m}$
 consisting of the bilinear operators $A(x,y)$ such that
 the supertrace of every $T_a$ is zero for all $a$, where $T_a(x)=A(a,x).$
 \smallskip

{\bf\underline{The superalgebra $H_{n,\,m}$}} ($n$ even) is the
 subsuperalgebra of $W_{n,\,m}$ consisting of the bilinear operators
``preserving'' a nondegenerate skew-symmetric bilinear consistent
 superform $\left<\cdot,\cdot \right>$:
 $$ \left<A(x,y),z \right> =(-1)^{yz}
 \left<A(x,z),y \right>.$$

All assertions and calculations made for $W_{n,\,m}$
hold also for $S_{n,\,m}$ and $H_{n,\,m}$,
except that the latter two superalgebras have no left units.

\subsection{The Main Theorem}
Let $M$ be a conservative superalgebra on a space $V$ with the
Jacobi subspace $J.$ Consider the space $W,$ which we define as
$W=V/J$ if $M$ has a left quasiunity, and $W=V/J\oplus E$ in the
opposite case, where $E$ is the one-dimensional even space with a
basis element $\epsilon.$

Assume that $M$ possesses a quasiunity. Define the {\it adjoint
mapping} $\operatorname{ad}: M \to \mathcal{U}(W)$ as follows:
\begin{eqnarray} \operatorname{ad}(a)( \alpha, \beta) =
(-1)^{\beta(a+\alpha)} ((\beta * a) * \alpha +
(-1)^{\alpha a}\beta * (\alpha a)-(-1)^{\alpha a} (\beta * \alpha)* a).
\label{23}
\end{eqnarray}
 If $M$ does not have a quasiunity,
 we define the adjoint mapping
$\operatorname{ad}: M \to \mathcal{U}(W)$ by the equation above and
the following equations:
\begin{eqnarray} \operatorname{ad}(a)(\alpha, \epsilon)
= a * \alpha + (-1)^{\alpha a}\alpha a - (-1)^{\alpha a}\alpha * a,
\label{24} \\
\operatorname{ad}(a)(\epsilon, \beta) =
(-1)^{a\beta}\beta * a,\ \operatorname{ad}(a)(\epsilon, \epsilon)=a.
\label{25}
\end{eqnarray}

We check that this mapping is well-defined. Firstly, we prove
that if $\alpha \in J$ or $\beta \in J$ then
$\operatorname{ad}(a)(\alpha,\beta) \in J.$ Indeed, if $\beta \in J,$
then the correctness  follows easily from (\ref{8}). If $\alpha \in
J,$ then by (\ref{9})
\[\operatorname{ad}(a)(\alpha,\beta) \equiv
(-1)^{a(\alpha+\beta)}(-\alpha(\beta*a) +
(-1)^{\alpha\beta}\beta*(\alpha a) + (\alpha\beta)*a) \
(\operatorname{mod} J).\] Now,  (\ref{jacobi_2}) and
(\ref{jacobi_double_comm}) imply
\[[L_\alpha,[L_a,[L_\beta,M]]] = [[L_\alpha,L_a],[L_\beta,M]] +
(-1)^{\alpha a}[L_a,[L_\alpha,[L_\beta,M]]]=\]
\[[L_{\alpha a},[L_\beta,M]] + (-1)^{\alpha a}[L_a,[L_{\alpha \beta},M]]
= -[L_{(-1)^{\beta(\alpha+a)}\beta*(\alpha a) +
(-1)^{a\beta}(\alpha \beta)*a},M].\]
On the other hand,
\[[L_\alpha,[L_a,[L_\beta,M]]] = -(-1)^{a\beta}[L_\alpha,[L_{\beta*a},M]]
= -(-1)^{a\beta}[L_{\alpha(\beta*a)},M].\] Comparing these
expressions and using  (\ref{jacobi_3}), we infer that
$\operatorname{ad}(a)(\alpha,\beta) \in J.$ The correctness of
 (\ref{24}) and (\ref{25}) follows from (\ref{8}) and (\ref{9}).

Note that if $M$ has a quasiunity $e,$ then the uniquely defined
element $\epsilon = e\, (\operatorname{mod} J)$ satisfies (\ref{24})
and  (\ref{25}). Indeed, (\ref{24}) follows from (\ref{10}). To
prove (\ref{25}) we note that
\[\operatorname{ad}(a)(e,\beta) \equiv
(-1)^{a\beta}(-e(\beta*a) + \beta*(ea) + (e\beta)*a) \
(\operatorname{mod} J).\] Now,  (\ref{quasiunit_oper}),
(\ref{quasiunit_double_comm}), and the Jacobi identity allow us to
rewrite the expression $[L_e,[L_a,[L_\beta,M]]]$ in two different
ways:
\[[L_e,[L_a,[L_\beta,M]]] = [[L_e,L_a],[L_\beta,M]] +
 [L_a,[L_e,[L_\beta,M]]]=\]
\[[L_{ea-a},[L_\beta,M]] + [L_a,[L_{e\beta-2\beta},M]] =
-(-1)^{a\beta}[L_{\beta*(ea) + (e\beta)*a - 3\beta*a},M],\]
and
\[[L_e,[L_a,[L_\beta,M]]] = -(-1)^{a\beta}[L_e,[L_{\beta*a},M]] =
-(-1)^{a\beta}[L_{e(\beta*a) - 2(\beta*a)},M].\] Comparing these
expressions, we obtain $\operatorname{ad}(a)(e,\beta) \equiv
(-1)^{a\beta}\beta*a.$

Moreover, it is easy to check that the mapping $\operatorname{ad}$
 does not depend on the associated multiplication $*.$
Indeed, let $*_1$ and $*_2$ be two associated multiplications on
$M.$
 By (\ref{8}) and the inclusion
 $\alpha *_1 \beta -\alpha *_2 \beta\in J$ which holds for all
 $\alpha, \beta \in M$ by (\ref{cons_oper}), we have
\[(\beta *_1 a) *_1 \alpha + (-1)^{\alpha a}\beta *_1 (\alpha a)-
(-1)^{\alpha a} (\beta *_1 \alpha)*_1 a -\]
\[ ((\beta *_2 a) *_2 \alpha + (-1)^{\alpha a}\beta *_2 (\alpha a)-
(-1)^{\alpha a} (\beta *_2 \alpha)*_2 a)\in J\]
 for all  $a, \alpha, \beta \in M$.
 We proceed analogously with (\ref{24}) and  (\ref{25}).

Prove that the adjoint mapping $\operatorname{ad}: M \to
(\mathcal{U}(W),\bigtriangleup_a)$ is a homomorphism for certain
 $a \in W.$ We begin with

\begin{Lem}
Let $(U_{-1},U_0,U_{1})$ be a triple of superspaces. Given some
$($\!superanticommutative$)$ commutators
\[[U_{-1},U_0] \subseteq U_{-1},\ [U_0,U_0] \subseteq U_0,\ [U_1,U_0]
\subseteq U_1,\ [U_{-1},U_1] \subseteq U_0,\] for every $a \in
U_{\pm 1}$ we define an algebra structure $M_a$ on $U_{\mp 1}$
$($\!whose parity coincides  with one of $a)$  by the rule
\[M_a(x,y) = [[a,x],y],\ x, y \in U_{\mp1}.\]
Assume that for the commutators above the Jacobi superidentity holds
whenever defined. Then for every even $a \in U_{-1}$ the mapping
\begin{gather*}
(U_1,M_a) \to (\mathcal{U}(U_{-1}),\bigtriangleup_a),\\
x \mapsto -M_x,
\end{gather*}
is an algebra homomorphism.
\end{Lem}
\begin{proof}
A direct computation shows that for $b, c \in U_{-1}$ we have
\begingroup\small\[((-M_x)\bigtriangleup_a(-M_y))(b,c) = [[x,a],[[y,b],c]] -
(-1)^{xy}[[y,[[x,a],b]],c] - (-1)^{x(y+b)}[[y,b],[[x,a],c]],\]\endgroup{}
\[-M_{M_a(x,y)}(b,c) = [[[[x,a],y],b],c],\]
and these expressions are equal by the Jacobi identity.
\end{proof}

\begin{Rem}
It is possible to construct a (unique in a sense)
$\mathbb{Z}$-graded Lie superalgebra $\mathcal{L} =
\sum_{i=-\infty}^{\infty}\mathcal{L}_i$ such that $\mathcal{L}_i =
U_i,\ i = -1, 0, 1$, such that the commutators on $\mathcal{L}_i$
coincide with the commutators on $U_i$ (see \cite{Kantor81}), but we
will not need this superalgebra here.
\end{Rem}

\begin{Th}\label{thm_main}
Let $M$ be a conservative superalgebra on a vector space $V$ with the
Jacobi subspace $J.$ Let either $W=V/J$ or $W=V/J\oplus
\left<\epsilon\right>$ as above.  The adjoint mapping
$\operatorname{ad}: M \to (\mathcal{U}(W),
\bigtriangleup_{-\epsilon})$ is a homomorphism whose kernel is
 the maximal Jacobi ideal. In particular, if $V$ is
finite-dimensional and $J$ is of codimension $n+m,$ then we have a
homomorphism $\operatorname{ad}: M \to \mathcal{U}(k,m),$ where $k =
n$ if $M$ has a quasiunity and $k = n+1$ otherwise.
\end{Th}

\begin{proof}
As above, take
\begin{eqnarray*}
\mathcal{U}_0(M) &:=& \langle L_a, [L_a,L_b] : a, b \in V
\rangle, \\
\mathcal{U}_1(M) &:=& \langle M, [L_a,M]: a \in V \rangle.
\end{eqnarray*}
Consider the triple of spaces $(\mathcal{U}_1(M),\, \mathcal{U}_0(M),\,
V)$ and define the commutators among them by
(\ref{oper_comm1}), (\ref{oper_comm2}), and the usual commutation in
$\mathfrak{gl}(V)$ (recall that  $\mathcal{U}_0(M)$ is a Lie
superalgebra that acts on $\mathcal{U}_1(M)$ and $V$). Then one can
easily check that the commutators (\ref{oper_comm1}) and
(\ref{oper_comm2}) satisfy the Jacobi super-identity whenever defined:
the only case not checked above is of the double commutator of the elements $A \in
\mathfrak{gl}(V),\, B \in \mathcal{U}(V),\, x \in V,$ and it can be verified
directly.

Therefore, we are under the conditions of the previous lemma. Take
 $M \in \mathcal{U}_1(M)$. We have a homomorphism $M
\to(\mathcal{U}(\mathcal{U}_1(M)),\bigtriangleup_M)$.
 Calculate the bilinear map $-M_a$ explicitly:
\[-M_a([L_{\alpha},M],[L_{\beta},M]) =
-[[a,[L_{\alpha},M]],[L_{\beta},M]]=
-[[L_a,L_{\alpha}]+(-1)^{\alpha a}L_{\alpha a}, [L_{\beta}, M]]=\]
\[-[L_a,[L_{\alpha}, [L_{\beta},M]]] +
(-1)^{\alpha a} [L_{\alpha}, [L_a, [L_{\beta}, M]]] -
(-1)^{\alpha a}[L_{\alpha a}, [L_{\beta}, M]]=\]
\[[L_{-(-1)^{\beta\alpha +
(\beta+\alpha)a}(\beta*\alpha)*a +
(-1)^{a\alpha+a\beta+(a+\beta)\alpha}(\beta*a)*\alpha +
(-1)^{a\alpha+\beta(a+\alpha)}\beta*(\alpha a)},M] =
[L_{\operatorname{ad}(a)(\alpha,\beta)},M],\]
\begingroup\small\[-M_a(M,[L_{\beta},M]) = -[[a,M],[L_{\beta},M]] =
[L_a,[L_{\beta},M]] = [L_{(-1)^{\beta a}\beta*a},-M] =
 [L_{\operatorname{ad}(a)(\epsilon,\beta)},-M],\]\endgroup{}
\[-M_a([L_{\alpha},M],M) = -[[a,[L_{\alpha},M]],M] =
-[[L_a,L_\alpha]+(-1)^{a\alpha}L_{\alpha a},M] =
[L_{\operatorname{ad}(a)(\alpha,\epsilon)},-M]\]
\[M_a(M,M) = -[[a,M],M] = [L_a,M].\]

Consider the mapping $\psi: V \to \mathcal{U}_1(M)$ given by $x
\mapsto [L_x,M].$ Then (\ref{jacobi_3}) means exactly that
$\ker(\psi) = J.$ Therefore, we have an injective map $\bar{\psi}:
V/J \to \mathcal{U}_1(M).$ Defining, if needed,
$\bar{\psi}(\epsilon) = - M$ (if $M$ has a left quasiunity $e$ then
$[L_e,M] = -M$ by (\ref{quasiunit_M})) we have an isomorphism
between $W$ and $\mathcal{U}_1(M).$ This induces an isomorphism
between $(\mathcal{U}(\mathcal{U}_1(M)),\bigtriangleup_M)$ and
$(\mathcal{U}(W),\bigtriangleup_{-\epsilon})$ given by
(\ref{induced_iso}). Now, by the formulas above
  this homomorphism in composition with the homomorphism
  which was constructed
 above is exactly the adjoint mapping.

Now we show that $\ker(\operatorname{ad}) = I, $ where $I$ is
 the maximal Jacobi
 ideal of $M$. Indeed, if $a \in I,$ then it follows from (\ref{8})
 and  (\ref{9}) that $\operatorname{ad}(a)=0.$

Conversely, if $\operatorname{ad}(a)=0,$ then $a \in J$ by  (\ref{25}).
 Since $\operatorname{ad}$ is a homomorphism,
  $\ker\operatorname{ad} \subseteq I,$ and
  the theorem is proved.\end{proof}

Consider $(\mathcal{U}(V),\bigtriangleup_a)$ for a fixed nonzero
$a\in V$. Check that the adjoint homomorphism, which is applied to
$\mathcal{U}(V)$, is the identity mapping. We know
  that $\mathcal{U}(V)$ has a left unity, and the maximal
Jacobi ideal of $\mathcal{U}(V)$  is zero.
 Therefore, in our case the space $W$ is
$\mathcal{U}(V)/J$ that we identify with $V$ by the mapping $A
\mapsto A(a,a).$ Now, let $A, B, C \in \mathcal{U}(V),$ and let
$B(a,a) = u, C(a,a) = v.$ Show that $\operatorname{ad}(A)(B,C)(a,a)
= A(u,v),$ which establishes the required isomorphism. Recall that
the adjoint mapping does not depend on the choice of associated
multiplication. Therefore, we choose it as (\ref{U_assmult_1}).
Now, a direct computation shows the desired equality:
\[\operatorname{ad}(A)(B,C)(a,a) = \]
\[(-1)^{C(A+B)}((C\bigtriangledown^1_a A)\bigtriangledown^1_a B +
(-1)^{BA}C\bigtriangledown^1_a(B\bigtriangleup_a A) -
(-1)^{BA}(C\bigtriangledown^1_a B)\bigtriangledown^1_a A)(a,a)=\]
\[(-1)^{C(A+B)}\Big((-1)^{B(A+C) + AC}B(a,A(a,C(a,a))) -
(-1)^{BA+C(B+A)}\big(B(a,A(a,C(a,a))) - \]
\[(-1)^{AB}A(B(a,a),C(a,a)) - (-1)^{AB}A(a,B(a,C(a,a)))\big) -
\] \[(-1)^{BA+A(C+B)+CB}A(a,B(a,C(a,a)))\Big) =
A(B(a,a),C(a,a)) = A(u,v). \]
In particular, the equality $\operatorname{ad}= id$ gives another proof
 of the fact that $\mathcal{U}(V)$ has no nonzero Jacobi ideals.

\appendix
\section*{Appendix A. Differences from the published version}
\addcontentsline{toc}{section}{Appendix A. Differences from the published version}

\noindent Some typos have been corrected.

\end{document}